\documentclass[12pt]{amsart}
\usepackage{amsmath,amssymb}

\DeclareMathOperator{\Aut}{Aut}

\begin{document}
\newtheorem{thm}{Theorem}[section]
\newtheorem{lem}[thm]{Lemma}
\newtheorem{dfn}[thm]{Definition}
\newtheorem{cor}[thm]{Corollary}
\newtheorem{conj}[thm]{Conjecture}
\newtheorem{clm}[thm]{Claim}
\newtheorem{que}[thm]{Question}
\newtheorem{prop}[thm]{Proposition}

\theoremstyle{remark}
\newtheorem{exm}[thm]{Example}
\newtheorem{rem}[thm]{Remark}

\def\N{{\mathbb N}}
\def\G{{\mathbb G}}
\def\F{{\mathbb F}}
\def\Q{{\mathbb Q}}
\def\R{{\mathbb R}}
\def\C{{\mathbb C}}
\def\P{{\mathbb P}}
\def\A{{\mathbb A}}
\def\Z{{\mathbb Z}}
\def\v{{\mathbf v}}
\def\w{{\mathbf w}}
\def\x{{\mathbf x}}
\def\O{{\mathcal O}}
\def\M{{\mathcal M}}
\def\kbar{{\bar{k}}}
\def\tr{\mbox{Tr}}
\def\id{\mbox{id}}

\renewcommand{\theenumi}{\alph{enumi}}

\title[conjugacy class of an isometry in orthogonal groups]{Single conjugacy classes of isometries in orthogonal groups over local fields}

\author{Fei Xu}

\author{Bo Zhang}

\thanks{The first author is partially supported by the National Key R $\And$ D Program of China No. 2023YFA1009702 and National Natural Science Foundation of China No. 12231009}

\begin{abstract} All isometries $\sigma$ in a quadratic space over a non-archimedean local field of characteristic not 2 satisfying that any isometry $\tau$ which is conjugate to $\sigma$ in the general linear group is conjugate to $\sigma$ in the orthogonal group are determined. This extends \cite[Theorem 2.1]{Mil} to arbitary cases.
\end{abstract}

\maketitle

\vspace{-5pt}
\section{Introduction}

Motivated by studying knot cobordism in \cite{Le}, Milnor developed classification theory of conjugacy classes of orthogonal groups in \cite{Mil} and proved that two isometries of a quadratic space $(V, \langle \  , \  \rangle)$ over a local field $k$ of $char(k)\neq 2$ with the same irreducible minimal polynomial  are conjugate in ${\rm O}(V)$ in \cite[Theorem 2.1]{Mil}, which was conjectured by Levine.  It is natural to have the following more general question.

\begin{que}\label{conj}
Determine all isometries $\sigma \in {\rm O}(V)$ satisfying that any isometry in ${\rm O}(V)$ which is conjugate to $\sigma$ in ${\rm GL}(V)$ is conjugate to $\sigma$ in ${\rm O}(V)$. 
\end{que}

When $\sigma$ is semi-simple and $k$ is a non-archimedean local field, an answer to Question \ref{conj} is given in \cite{CHLX} by using the moduli spaces of  pairs of a quadratic space with a semi-simple isometry. In this paper, we give a complete answer to Question \ref{conj} over a non-archimedean local field by applying Milnor's theory for classification conjugacy classes of orthogonal groups in \cite{Mil}. 

\medskip

Let $f(x)$ be the characteristic polynomial of $\sigma\in  {\rm O}(V)$. Then 
 \begin{equation} \label{factor} f(x)=(x+1)^{m_+} (x-1)^{m_-} \prod_{i=1}^{m_1} p_i(x)^{e_i}  \prod_{j=1}^{m_2} (\tilde{p}_j(x) \tilde{p}_j^*(x))^{\tilde{e}_j} \end{equation}
  by \cite[7. Lemma (a)]{Le}, where $p_i(x)$ is a monic irreducible polynomial of degree bigger than 1 over $k$ such that $$p_i(x)=x^{\deg (p_i)} p_i(x^{-1}) \ \ \ \text{ for $1\leq i\leq m_1$} $$  and $\tilde{p}_j(x)$ is a monic irreducible polynomial over $k$ and 
$$ \tilde{p}_j^*(x) =\tilde{p}_j(0)^{-1} x^{\deg(\tilde{p}_j)} \tilde{p}_j(x^{-1})  \ \ \ \text{with} \ \ \ \tilde{p}_j(x) \neq \tilde{p}_j^*(x)$$ for $1\leq j\leq m_2$. 

For a monic irreducible factor $q(x)$ of $f(x)$, one defines 
$$ V_{q(x), \sigma} =\{ v\in V: \ q(\sigma)^l v =0   \ \text{for sufficiently large integer $l$} \} . $$

When $q(x)=x\pm 1$ or $p_i(x)$ for $1\leq i\leq m_1$, then there is an orthogonal decomposition
$$ V_{q(x), \sigma } = V_{q(x), \sigma }^{(1)} \perp \cdots \perp V_{q(x), \sigma }^{(l)} \perp \cdots $$
where  $V_{q(x), \sigma}^{(l)}$ is a free $k[x]/(q(x)^l)$-module by \cite[Theorem 3.2]{Mil}.

For $q(x)=x \pm 1$, the well-defined bilinear map 
\begin{equation} \label{bil}  b_{\sigma}^{(l)} :  (V_{q(x), \sigma}^{(l)}/q(x)V_{q(x), \sigma}^{(l)} )  \times (V_{q(x), \sigma}^{(l)}/q(x)V_{q(x), \sigma}^{(l)}) \rightarrow k \end{equation}
 $$ b_{\sigma}^{(l)}(\bar u, \bar v) = \langle (\sigma-\sigma^{-1})^{l-1} u, v \rangle $$
is non-degenerate, symmetric when $l$ is odd by \cite[Theorem 3.4]{Mil}.

The main result of this paper gives a complete answer to Question \ref{conj}  over a non-archimedean local field, which extends \cite[Theorem 2.1]{Mil} to arbitrary cases.

\begin{thm} \label{main} Let $V$ be a non-degenerated quadratic space over a non-archimedean local field $k$ with $char(k)\neq 2$. Suppose $\sigma\in {\rm O}(V)$ and $f(x)$ is the characteristic polynomial of $\sigma$ with the factorization of (\ref{factor}). Write 
$$ m_0= \sharp \{ 2l+1 : V_{x+1, \sigma}^{(2l+1)} \neq 0  \} + \sharp \{ 2l+1: V_{x-1, \sigma}^{(2l+1)} \neq 0 \} . $$

Then any isometry in ${\rm O}(V)$ which is conjugate to $\sigma$ in ${\rm GL}(V)$ is conjugate to $\sigma$ in ${\rm O}(V)$ if and only if  one of the following conditions holds

(i) $m_0\leq 1$ and $m_1=0$.

(ii) $m_0=0$ and $m_1=1$ and $V_{p_1(x), \sigma}=V_{p_1(x), \sigma}^{(2l_1+1)}$ for some integer $l_1\geq 0$.   

(iii) $m_0=m_1=1$ and $V_{p_1(x), \sigma}=V_{p_1(x), \sigma}^{(2l_1+1)}$ for some integer $l_1\geq 0$ and the unique non-zero quadratic space $$(V_{x\pm 1, \sigma}^{(2l_0+1)}/(x\pm 1)V_{x\pm 1, \sigma}^{(2l_0+1)} , b_\sigma^{(2l_0+1)})$$ defined by (\ref{bil}) is either one dimensional or a hyperbolic plane.
\end{thm}

Terminology and notation are standard if not explained and are adopted from \cite{OM} and \cite{Mil}. Let $k$ be a field $k$ of $char(k)\neq 2$.  A non-degenerate quadratic space $(V,  \langle \ , \  \rangle )$ is a finite dimensional vector space $V$ over $k$ with a symmetric bilinear map $$  \langle \ , \  \rangle : \  V\times V \longrightarrow k \ \ \ \text{with} \ \ \ \det(\langle e_i, e_j \rangle) \neq 0 $$ for a basis $\{ e_1, \cdots, e_n \}$ of $V$.  Write $$ {\rm O}(V)=\{ \sigma \in {\rm GL}(V): \  \langle \sigma u, \sigma v \rangle  =  \langle u, v \rangle ,  \ \forall u, v \in V \} . $$
A non-degenerate two dimensional quadratic space $(V, \langle \  , \  \rangle)$ is called a hyperbolic plane if there is $v\in V$ with $v\neq 0$ such that $\langle  v , v \rangle =0$. 
If $f(x)$ is the characteristic polynomial of $\sigma\in {\rm O}(V)$ and $q(x)$ is a monic irreducible factor of $f(x)$, we write 
$$ V_{q(x), \sigma} =\{ v\in V: \ q(\sigma)^l v =0   \ \text{for sufficiently large integer $l$} \} . $$ 

A separable $k$-algebra is defined to be a finite dimensional semi-simple $k$-algebra such that the center of each simple component is a separable field extension of $k$. For a separable $k$-algebra $A$, the reduced trace of $A$ to $k$ is denoted by ${\rm Tr}_{A/k}$ (see \cite[\S 9]{Reiner}).  For a polynomial $f(x)$ of degree $n$ over $k$ with $f(0)\neq 0$, we define $$f^*(x)= f(0)^{-1} x^n f(x^{-1}) . $$

\section{Hermitian structures for non-reciprocal cases}

It is well-known that the classification of conjugacy classes of ${\rm GL}(V)$ is equivalent to classify all $k[x]$-module structure of $V$ via $\sigma \in {\rm GL}(V)$ for a finite dimensional  vector space $V$ over $k$.  When $V$ is a quadratic space over $k$, the classification of conjugacy classes of ${\rm O}(V)$ becomes classification of $k[x]$-module structure of $V$  via $\sigma \in {\rm O}(V)$ compatible with the quadratic structure. This leads to Milnor's classification in \cite[\S 3]{Mil}. A brief explanation for Case 3 in \cite[\S 3]{Mil} is given in the last paragraph in \cite[p.94]{Mil}. We will provide the similar hermitian structure for Case 3 in \cite[\S 3]{Mil} which makes this theory in a uniform way. 

\begin{dfn} Let $p(x)$ be an irreducible monic polynomial over $k$ with $p(0)\neq 0$ such that $p(x)\neq p^*(x)$ and $E=k[x]/(p(x)p^*(x))$. 

A free $E$-module $M$ of finite rank is called a hermitian space over $E$ if there is a map 
$$ h: M\times M \rightarrow E; \ (u, v) \mapsto h(u, v)$$  such that $h(u, v)$ is linear in $u$ and $h(v, u) = \iota (h(u, v))$ where $\iota$ is an involution of $E$ over $k$ induced by $x\mapsto x^{-1}$. 

A hermitian space $(M, h)$ over $E$ is called non-degenerate if  
$$ h(u, v)= 0 , \ \forall v\in M \ \ \ \Leftrightarrow \ \ \ u=0. $$

The unitary group $U(M, h)$ of $(M, h)$ is defined by 
$$ {\rm U}(M, h)= \{ \sigma \in \Aut_E(M): \ h(\sigma u, \sigma v) = h(u, v),  \ \forall u, v\in M \} $$
\end{dfn}

The following result gives the equivalent condition that a hermitian space is non-degenerate.

\begin{prop} \label{non-degenerate} Let $p(x)$ be an irreducible monic polynomial over $k$ with $p(0)\neq 0$ such that $p(x)\neq p^*(x)$ and $E=k[x]/(p(x)p^*(x))$. 
Suppose $(M, h)$ is a hermitian space over $E$. 

Then $(M, h)$ is non-degenerate if and only if $\det(h(e_i, e_j))_{1\leq i, j \leq n}$ is invertible in $E$ for a basis $\{ e_1, \cdots, e_n \}$ of $M$ over $E$.
\end{prop}

\begin{proof} ($\Rightarrow$) Suppose that $\det(h(e_i, e_j))_{1\leq i, j \leq n}$ is a zero divisor in $E$ for a basis $\{ e_1, \cdots, e_n \}$ of $M$ over $E$.  Since 
$$ E\cong k[x]/(p(x)) \times k[x]/(p^*(x)) , $$
one can write $$h(e_i, e_j)= (a_{ij}, b_{i j}) \ \ \text{ with } \ \ a_{ij}\in k[x]/(p(x)), \ b_{ij} \in  k[x]/(p^*(x))$$ for $1\leq i, j\leq n$. Since
$$\det(h(e_i, e_j))= (\det(a_{ij}), \det(b_{ij})) $$ is a zero divisor in $E$, one has $\det(a_{ij})=0$ or $\det(b_{ij})=0$. Without loss of generality, we assume $\det(a_{ij})=0$. Then  
there is $$(\xi_1, \cdots, \xi_n) \neq (0, \cdots, 0) \ \ \ \text{such that} \ \ \ (\xi_1, \cdots, \xi_n) \cdot (a_{ij}) = (0, \cdots, 0)$$ where $\xi_1, \cdots, \xi_n$ are all in $k[x]/(p(x))$.  Therefore
$$ u_0 = (\xi_1, 0) e_1 + \cdots + (\xi_n, 0) e_n  \neq 0. $$
For any $v= \alpha_1 e_1+ \cdots + \alpha_n e_n \in M$, one has 
$$ \begin{aligned} & h(u_0, v)  = ((\xi_1, 0), \cdots, (\xi_n, 0)) \cdot (h(e_i, e_j)) \cdot (\iota(\alpha_1), \cdots, \iota(\alpha_n))^{T} \\
& = ((\xi_1, \cdots, \xi_n)(a_{ij}), (0, \cdots, 0) (b_{ij}))\cdot (\iota(\alpha_1), \cdots, \iota(\alpha_n))^{T} =0 .
\end{aligned}$$ 
A contradiction is derived. 

($\Leftarrow$) Let $u=\beta_1 e_1+ \cdots + \beta_n e_n$ such that $h(u, v)=0$ for all $v\in M$. 
Then 
$$ (\beta_1, \cdots, \beta_n) (h(e_i, e_j))^T = (0, \cdots, 0) .$$ Since $\det(h(e_i, e_j))_{1\leq i, j \leq n}$ is invertible in $E$, one obtains  $$(\beta_1, \cdots, \beta_n) = (0, \cdots, 0) . $$ Therefore $u=0$ as required. 
\end{proof} 




This kind of non-degenerate hermitian space of the same rank is unique up to isomorphism. 

\begin{prop} \label{orth-base}  Let $p(x)$ be an irreducible monic polynomial over $k$ with $p(0)\neq 0$ such that $p(x)\neq p^*(x)$ and $E=k[x]/(p(x)p^*(x))$.  Assume that $char(k)\neq 2$. 

If $(M, h)$ is a non-degenerate hermitian space over $E$ with respect to the involution induced by $x\mapsto x^{-1}$, then $M$ has an orthogonal basis $\{ v_1, \cdots, v_n \}$ over $E$ such that $$ h(v_1, v_1) = \cdots = h(v_n, v_n) =1. $$ In particular,  any two non-degenerate hermitian spaces of the same rank over $E$ with respect to the involution induced by $x\mapsto x^{-1}$ are isomorphic. 
\end{prop}
\begin{proof} Let $\{ e_1, \cdots, e_n \}$ be a basis of $M$ over $E$. 

If $h(e_1, e_1)$ is invertible in $E$, then $M= E e_1 \perp N$ where 
$$ N= E(e_2- \frac{h(e_1, e_2)}{h(e_1, e_1)} e_1)+  \cdots + E(e_n- \frac{h(e_1, e_n)}{h(e_1, e_1)} e_1). $$
Since $\det (M) = \det (Ee_1) \cdot \det (N)$, one obtains that $N$ is also a non-degenerate hermitian space over $E$ by Proposition \ref{non-degenerate}. The result follows from induction.

Otherwise, one has that $h(e_1, e_1)$ is a zero divisor of $E$. Then there is $2\leq i_0\leq n$ such that $h(e_1, e_{i_0})\neq 0$ by the assumption that $\det (h(e_i, e_j))$ is invertible in $E$. Without loss of generality, we simply assume that $h(e_1, e_2)\neq 0$.  

Claim: There is $\lambda\in E$ such that $h(\lambda e_1+ e_2, \lambda e_1+  e_2)$ is invertible in $E$. 

Indeed, since 
$$ E\cong k[x]/(p(x)) \times k[x]/(p^*(x)) \cong K\times K $$
as $k$-algebras where $K=k(\alpha)$ and $\alpha$ is a root of $p(x)$, the involution of $E$ induced by $x\mapsto x^{-1}$ gives the involution of $K\times K$ by interchanging two coordinates.  Under this identification, one can write 
$$ h(e_1, e_1)= (a_1, a_2), \ \ \ h(e_1, e_2) = (c_1, c_2) \ \ \ \text{and} \ \ \ h(e_2, e_2) =(b_1, b_2)$$  
where $a_i, b_i, c_i\in K$ for $i=1, 2$.  Write $\lambda= (\lambda_1, \lambda_2)$. Then 
$$ h(\lambda e_1+ e_2, \lambda e_1+ e_2)=(a_1\lambda_1\lambda_2+c_1\lambda_1+c_2\lambda_2+b_1, a_2\lambda_1\lambda_2+ c_1\lambda_1+c_2\lambda_2+b_2).$$
Since $h(e_1, e_1)$ is a zero divisor of $E$, one can simply assume that $a_1=0$.  By $h(e_1, e_2)= (c_1, c_2)\neq (0, 0)$, the linear form $c_1\lambda_1+c_2\lambda_2$ is not zero. 

When $a_2\neq 0$, one considers the equation 
$$ c_1\lambda_1+c_2\lambda_2+b_1=1 .$$  Then the equation 
$$ a_2\lambda_1\lambda_2+ c_1\lambda_1+c_2\lambda_2+b_2=0 $$ has at most two solutions in $K$. Since $char(k)\neq 2$, the field $K$ contains more than two elements. Therefore there are $\lambda_1, \lambda_2$ in $K$ such that 
$$ a_2\lambda_1\lambda_2+ c_1\lambda_1+c_2\lambda_2+b_2 \neq 0 . $$ This implies that $h(\lambda e_1+ e_2, \lambda e_1+  e_2)$ is invertible in $E$.

When $a_2=0$, there are $\lambda_1, \lambda_2\in K$ such that $$c_1\lambda_1+c_2\lambda_2 \in K\setminus \{-b_1, -b_2 \}$$ by $char(k)\neq 2$. Therefore the claim follows.

By repeating the beginning arguments to the basis $\{ \lambda e_1+ e_2, e_1, e_3, \cdots, e_n \}$ of $M$, one obtains an orthogonal basis $\{\eta_1, \cdots, \eta_n \}$. Write
$$ h(\eta_1, \eta_1)= (\alpha_1, \alpha_1), \cdots, h(e_n, e_n)=(\alpha_n, \alpha_n) $$ with $\alpha_1, \cdots, \alpha_n \in K$. Then 
$$ v_1= (\alpha_1, 1)^{-1} \eta_1, \cdots, v_n = (\alpha_n, 1)^{-1} \eta_n $$ is as required. 
\end{proof}

\begin{cor} If $(M, h)$ is the non-degenerate hermitian space over $E$ in the above Proposition \ref{orth-base}, then 
$${\rm U}(M, h) \cong {\rm GL}_n (K) $$ 
where $n={\rm rank}_E (M)$ and $K=k(\alpha)$ for a root  $\alpha$ of $p(x)$.
\end{cor}
\begin{proof}  Suppose that  $\{ v_1, \cdots, v_n \}$ is an orthogonal basis over $E$ such that $ h(v_1, v_1) = \cdots = h(v_n, v_n) =1$. Let $\sigma  \in \Aut_E(M)$ and write 
$$ \sigma(v_i) = \sum_{j=1}^n (a_{i j}, b_{i j})\cdot  v_j $$
where $a_{ij}, b_{ij}\in K$ for $1\leq i\leq n$.  Then 
$$ \sigma\in {\rm U}(M, h) \ \ \ \Leftrightarrow \ \ \ ((a_{ij}, b_{i j})) \cdot ((b_{j i}, a_{j i}) )= I_n .$$
This is equivalent to $(a_{ij})=(b_{ji})^{-1} \in {\rm GL}_n (K) $.  The map $\sigma \mapsto (a_{ij})$ gives the desired isomorphism.
\end{proof}

One also has an analogue of \cite[Theorem 3.2]{Mil} for Case 3 in \cite[\S 3]{Mil}. 

\begin{thm}  \label{splitting} Let $(V,  \langle \  , \  \rangle)$ be a non-degenerate quadratic space over a field of $k$ with $char(k)\neq 2$ and $f(x)$ be the characteristic polynomial of $\sigma\in {\rm O}(V)$.  If $\tilde{p}(x)$ is a monic irreducible factor of $f(x)$ with
$ \tilde{p}(x) \neq \tilde{p}^*(x) $,  then there is a direct sum decomposition
$$ V_{\tilde{p}(x), \sigma} = V_{\tilde{p}(x), \sigma}^ {(1)}\oplus \cdots \oplus V_{\tilde{p}(x), \sigma}^ {(i)}  \cdots $$
where $V_{\tilde{p}(x), \sigma}^ {(i)}$ is a free $k[x]/(\tilde{p}(x)^i)$-module and 
$$ V_{\tilde{p}^*(x), \sigma} = V_{\tilde{p}^*(x), \sigma}^ {(1)}\oplus \cdots \oplus V_{\tilde{p}^*(x), \sigma}^ {(i)}  \cdots $$ 
where $V_{\tilde{p}^*(x), \sigma}^ {i}$ is  a free $k[x]/(\tilde{p}^*(x)^i)$-module  such that 
$$ V_{\tilde{p}(x), \sigma} \oplus V_{\tilde{p}^*(x), \sigma}  = (V_{\tilde{p}(x), \sigma}^ {(1)}\oplus V_{\tilde{p}^*(x), \sigma}^ {(1)}) \perp \cdots \perp (V_{\tilde{p}(x), \sigma}^ {(i)} \oplus V_{\tilde{p}^*(x), \sigma}^ {(i)}) \perp \cdots $$
where $V_{\tilde{p}(x), \sigma}^ {(i)} \oplus V_{\tilde{p}^*(x), \sigma}^ {(i)}$ is a free $k[x]/(\tilde{p}(x)^i \tilde{p}^*(x)^i)$-module. 
\end{thm}

\begin{proof}  Let 
$$ {\rm Ann}_{k[x]}(V_{\tilde{p}(x), \sigma})= \{ a(x)\in k[x]:   a(x) V_{\tilde{p}(x), \sigma} =0 \} =(\tilde{p}(x)^s) $$
for some positive integer $s$ and 
$$ {\rm Ann}_{k[x]}(V_{\tilde{p}^*(x), \sigma})= \{ a(x)\in k[x]:   a(x) V_{\tilde{p}^*(x), \sigma} =0 \} =(\tilde{p}^*(x)^t) $$
for some positive integer $t$. Then $s=t$. Indeed, suppose $s>t$. Since there is $e\in V_{\tilde{p}(x), \sigma} $ such that $\tilde{p}(x)^{s-1} e\neq 0$ and $ V_{\tilde{p}(x), \sigma} \oplus V_{\tilde{p}^*(x), \sigma} $ is a non-degenerate subspace of $V$, there is $f\in V_{\tilde{p}^*(x), \sigma}$ such that $$ \langle  \tilde{p}(x)^{s-1} e ,  f \rangle =1 $$ by \cite[Lemma 3.1]{Mil}. Then 
$$1 = \langle  \tilde{p}(x)^{s-1} e ,  f \rangle =\langle  e ,  \tilde{p}^*(x)^{(s-1)} \tilde{p}^*(0)^{s-1} \sigma^{(1-s)\deg(\tilde{p}(x))}(f) \rangle = \langle  e ,  0 \rangle =0 $$
and a contradiction is derived. 

By the structure theorem of finitely generated modules over PID, there are a free $k[x]/(\tilde{p}(x)^s)$-module $V_{\tilde{p}(x), \sigma}^ {(s)}$ and a free $k[x]/(\tilde{p}^*(x)^s)$-module $V_{\tilde{p}^*(x), \sigma}^ {(s)} $ such that 
$$ V_{\tilde{p}(x), \sigma} = V_{\tilde{p}(x), \sigma}' \oplus V_{\tilde{p}(x), \sigma}^ {(s)}  \ \ \ \text{and} \ \ \ 
 V_{\tilde{p}^*(x), \sigma} = V_{\tilde{p}^*(x), \sigma}'\oplus V_{\tilde{p}^*(x), \sigma}^ {(s)} $$ 
where  $V_{\tilde{p}(x), \sigma}'$ and $V_{\tilde{p}^*(x), \sigma}' $ are submodules of $ V_{\tilde{p}(x), \sigma} $ and $V_{\tilde{p}^*(x), \sigma}$ respectively satisfying $$\tilde{p}(x)^{s-1} V_{\tilde{p}(x), \sigma}'=\tilde{p}^*(x)^{s-1} V_{\tilde{p}^*(x), \sigma}' =0 . $$  Then $V_{\tilde{p}(x), \sigma}^ {(s)} \oplus V_{\tilde{p}^*(x), \sigma}^ {(s)} $ is a non-degenerate quadratic space over $k$. Indeed, suppose there is  $$v_0\in V_{\tilde{p}(x), \sigma}^ {(s)}  \ \ \text{ with }  \ \ v_0\neq 0 \ \ \text{ such that } \ \ \langle v_0, v \rangle =0, \  \forall  v\in V_{\tilde{p}^*(x), \sigma}^ {(s)} . $$ Write 
$$ (\tilde{p}(x)^m)={\rm Ann}_{k[x]} (v_0)= \{ a(x)\in k[x]: a(x) v_0=0 \}. $$
Then there is $w_0\in V_{\tilde{p}(x), \sigma}^ {(s)} $ such that $\tilde{p}(x)^{m-1}v_0= \tilde{p}(x)^{s-1}w_0\neq 0$.  Since $$\tilde{p}(\sigma^{-1})^{m-1} v\in  V_{\tilde{p}^*(x), \sigma}^ {(s)}  \  \ \ \text{for} \ \ \  v\in V_{\tilde{p}^*(x), \sigma}^ {(s)} ,$$ one obtains
$$\langle \tilde{p}(x)^{m-1}v_0, v \rangle = \langle v_0,  \tilde{p}(\sigma^{-1})^{m-1} v \rangle =0 $$ for all $v\in V_{\tilde{p}^*(x), \sigma}^ {(s)}$.  Since $$\tilde{p}^*(0)^{s-1} \sigma^{(1-s)\deg(\tilde{p}(x))}(w) \in V_{\tilde{p}^*(x), \sigma}' $$ for any $w\in V_{\tilde{p}^*(x), \sigma}' $, one has 
$$ \begin{aligned} & \langle \tilde{p}(x)^{m-1}v_0, w \rangle = \langle \tilde{p}(x)^{s-1}w_0, w \rangle \\
= \ & \langle w_0, \tilde{p}^*(x)^{s-1} (\tilde{p}^*(0)^{s-1} \sigma^{(1-s)\deg(\tilde{p}(x))}(w)) \rangle= \langle w_0, 0\rangle =0 \end{aligned} $$
for any $w\in V_{\tilde{p}^*(x), \sigma}' $. Since $V$ is non-degenerate, one concludes that  $$\tilde{p}(x)^{m-1}v_0= \tilde{p}(x)^{s-1}w_0= 0$$ by \cite[Lemma 3.1]{Mil}.  A contradiction is derived. Therefore 
$$ V_{\tilde{p}(x), \sigma} \oplus V_{\tilde{p}^*(x), \sigma}  = (V_{\tilde{p}(x), \sigma}^ {(s)}\oplus V_{\tilde{p}^*(x), \sigma}^ {(s)}) \perp (W_{\tilde{p}(x), \sigma} \oplus W_{\tilde{p}^*(x), \sigma} ) $$
where 
$$ W_{\tilde{p}(x), \sigma}=\{ v\in V_{\tilde{p}(x), \sigma}: \ \langle v, V_{\tilde{p}^*(x), \sigma}^ {(s)} \rangle =0 \}$$ and $$W_{\tilde{p}^*(x), \sigma} )= \{ v\in  V_{\tilde{p}^*(x), \sigma} : \ \langle v, V_{\tilde{p}(x), \sigma}^ {(s)} \rangle =0 \} . $$
The result follows from induction on $W_{\tilde{p}(x), \sigma} \oplus W_{\tilde{p}^*(x), \sigma}$.  \end{proof}

The following lemma is an analogue of \cite[Lemma 1.1]{Mil} for Case 3 in \cite[\S 3]{Mil} which is also given in \cite[Remark 2.10]{CHLX}. For convenience, we provide a proof which is close to the original proof \cite[Lemma 1.1]{Mil}.

\begin{lem} \label{hermitian} Let $(V,  \langle \  , \  \rangle)$ be a non-degenerate quadratic space over a field of $k$ with $char(k)\neq 2$. Suppose $\sigma \in {\rm O}(V)$ such that the minimal polynomial of $\sigma$ is $\tilde{p}(x) \tilde{p}^*(x)$ where $\tilde{p}(x)$ is a monic irreducible separable polynomial with $\tilde{p}(x) \neq \tilde{p}^*(x) $. 

If $E=k[x]/(\tilde{p}(x) \tilde{p}^*(x))$ and $V$ is $E$-module via $\sigma$, then $V$ over $E$ admits one and only one hermitian inner product $h$ with respect to the involution $\iota$ of $E$ over $k$ induced by $x\mapsto x^{-1}$ such that $$ \langle u  ,  v  \rangle = {\rm Tr}_{E/k} (h(u, v)) \ \ \ \forall \ u, v\in V. $$
\end{lem}
\begin{proof} For any $u, v\in V$, one can define a linear map 
$$ l_{u, v}: \ E  \longrightarrow k; \ \lambda \mapsto  \langle \lambda u, v \rangle  $$ over $k$. Since $E/k$ is an etale algebra, the reduced trace map
$$ E \times E \longrightarrow k; \ \ (\alpha, \beta) \mapsto {\rm Tr}_{E/k} (\alpha \beta) $$ is non-degenerate by \cite[(9.9) Theorem]{Reiner}.
 Then there is a unique $$h(u, v)\in E \ \ \text{such that} \ \  
 l_{u, v} (\lambda)= \langle \lambda u, v \rangle = {\rm Tr}_{E/k} (h(u, v)  \lambda) $$ for all $\lambda \in E$. In particular, $\langle u, v \rangle = {\rm Tr}_{E/k} (h(u, v))$. 

We claim that 
$$ h: \ V \times V  \longrightarrow E ; \ (u, v)\mapsto h(u, v) $$ is a hermitian form over $E$ with respect to the 
 involution $\iota$ induced by sending $x\mapsto x^{-1}$. First of all, $V$ is a free $E$-module of finite rank by Theorem \ref{splitting}.  Since any $\lambda\in E$ can be written in terms of $x$ over $k$, one has
$$ \langle \lambda u, v \rangle = \langle u, \iota (\lambda) v \rangle = \langle \iota (\lambda) v, u \rangle  $$
for all $u, v\in V$.  This implies that 
$$ {\rm Tr}_{E/k} (h(u, v)  \lambda) = {\rm Tr}_{E/k} (h(v, u) \iota( \lambda)) = {\rm Tr}_{E/k} (\iota (h(v, u))  \lambda) $$ for all $\lambda\in E$. Therefore $h(u, v)  = \iota (h(v, u))$ as desired. Since $(V,  \langle \  , \  \rangle)$ is non-degenerate, one concludes that $h$ is also non-degenerate. 

Suppose there is another hermitian form $h'$ on $V$ over $E$ such that $$\langle u, v \rangle = {\rm Tr}_{E/k} (h(u, v))={\rm Tr}_{E/k} (h'(u, v))$$ for all $u, v\in V$. Then 
$${\rm Tr}_{E/k} (\lambda h(u, v))={\rm Tr}_{E/k} (\lambda h'(u, v))$$ for any $\lambda\in E$. This implies that $h=h'$ and the uniqueness property for $h$ holds. 
\end{proof}

The following result is an analogue of \cite[Theorem 3.3]{Mil} for Case 3 in \cite[\S 3]{Mil}. 

\begin{thm} \label{complete invariant} Let $(V,  \langle \  , \  \rangle)$ be a non-degenerate quadratic space over a field of $k$ with $char(k)\neq 2$ and $f(x)$ be the characteristic polynomial of $\sigma\in {\rm O}(V)$.  Suppose that  $\tilde{p}(x)$ is a monic irreducible separable factor of $f(x)$ with 
$ \tilde{p}(x) \neq \tilde{p}^*(x) $. If
$$ V_{\tilde{p}(x), \sigma} \oplus V_{\tilde{p}^*(x), \sigma}  = (V_{\tilde{p}(x), \sigma}^ {(1)}\oplus V_{\tilde{p}^*(x), \sigma}^ {(1)}) \perp \cdots \perp (V_{\tilde{p}(x), \sigma}^ {(i)} \oplus V_{\tilde{p}^*(x), \sigma}^ {(i)}) \perp \cdots $$
is an orthogonal splitting in Theorem \ref{splitting}, then the space 
$$ H_{\sigma}^{(i)}= (V_{\tilde{p}(x), \sigma}^ {(i)}/\tilde{p}(x) V_{\tilde{p}(x), \sigma}^ {(i)})  \oplus (V_{\tilde{p}^*(x), \sigma}^ {(i)} / \tilde{p}^*(x) V_{\tilde{p}^*(x), \sigma}^ {(i)} )$$ over $E=k[x]/(\tilde{p}(x) \tilde{p}^*(x))$ admits a unique hermitian inner product $h_{\sigma}^{(i)}$ with respect to the involution $\iota$ induced by $x\mapsto x^{-1}$ satisfying 
$$ \langle \tilde{p}(x)^{i-1} u +u^* , \tilde{p}(x)^{i-1} v + v^* \rangle = {\rm Tr}_{E/k} (h_{\sigma}^{(i)} (\overline{u+u^*}, \overline{v+v^*})) $$ for $u, v \in V_{\tilde{p}(x), \sigma}^ {(i)} $ and $u^*, v^* \in V_{\tilde{p}^*(x), \sigma}^ {(i)}$ with $i=1, 2, \cdots$.  

Moreover, the sequence of isomorphism classes of the hermitian spaces $(H_{\sigma}^{(i)}, h_{\sigma}^{(i)})$ forms a complete invariant of pair $(V_{\tilde{p}(x), \sigma} \oplus V_{\tilde{p}^*(x), \sigma} , \sigma) $. 
\end{thm}
\begin{proof} Define a bilinear map $b_{\sigma}^{(i)}$ on the vector space $H_{\sigma}^{(i)}$ over $k$ by 
$$  b_{\sigma}^{(i)} (\bar \xi + \bar \xi^*,  \bar \eta + \bar \eta ^* ) = \langle \tilde{p} (x)^{i-1} \xi  + \xi ^* , \   \tilde{p} (x)^{i-1} \eta+  \eta^* \rangle $$
where $\xi, \eta \in V_{\tilde{p}(x), \sigma}^ {(i)}$ and $\xi^*, \eta^* \in V_{\tilde{p}^*(x), \sigma}^ {(i)}$.  It can be verified that $b_{\sigma}^{(i)}$ is well-defined and symmetric. Suppose $\xi_0 \in V_{\tilde{p}(x), \sigma}^ {(i)}$ satisfying $\bar \xi_0\neq 0$ in $H^{(i)}$.  Then $\tilde{p} (x)^{i-1} \xi_0 \neq 0$. Since $V_{\tilde{p}(x), \sigma}^ {(i)} \oplus V_{\tilde{p}^*(x), \sigma}^ {(i)}$ is non-degenerate, there is $\eta_0^*\in V_{\tilde{p}^*(x), \sigma}^ {(i)}$ such that $\langle \tilde{p} (x)^{i-1} \xi_0, \eta_0^* \rangle \neq 0$. Therefore 
$$ b_{\sigma}^{(i)} (\bar \xi_0,  \bar {\eta_0^*}) =\langle \tilde{p} (x)^{i-1} \xi_0, \eta_0^* \rangle \neq 0 . $$ This implies that $b_{\sigma}^{(i)}$ is non-degenerate.  Since $\sigma$ induces an isometry of quadratic space $(H_{\sigma}^{(i)}, b_{\sigma}^{(i)})$ with the minimal polynomial $\tilde{p}(x) \tilde{p}^*(x)$,  one obtains the required hermitian form $h_{\sigma}^{(i)}$ by applying Lemma \ref{hermitian}. 
Since $(H_{\sigma}^{(i)}, b_{\sigma}^{(i)})$ is orthogonal sum of hyperbolic planes by \cite[Lemma 3.1]{Mil}, the hermitian space $(H_{\sigma}^{(i)}, h_{\sigma}^{(i)})$ does not depend on the choice of $V_{\tilde{p}(x), \sigma}^ {(i)} \oplus V_{\tilde{p}^*(x), \sigma}^ {(i)}$ up to isomorphism by Lemma \ref{hermitian}. 

Since the isomorphism classes of hermitian spaces $(H_{\sigma}^{(i)}, h_{\sigma}^{(i)})$ determine the ranks of free $k[x]/(p(x)^i)$-modules $V_{\tilde{p}(x), \sigma}^ {(i)} $  for $i=1, 2, \cdots$, the  sequence of isomorphism classes of hermitian spaces $(H_{\sigma}^{(i)}, h_{\sigma}^{(i)})$ determines the isomorphic class $V_{\tilde{p}(x), \sigma} $ as $k[x]$-module which is a complete invariant for the pair $(V_{\tilde{p}(x), \sigma} \oplus V_{\tilde{p}^*(x), \sigma} , \sigma)$.  
\end{proof}

\section{More complements about Milnor's classification of conjugacy classes of orthogonal groups}

Since a complete invariant for isomorphic classes of pair of a quadratic space with an isometry is given by a sequence of hermitian spaces by \cite[Theorem 3.3]{Mil},  it is natural to ask if any sequence of hermitian spaces can be realized as a complete invariant of some pair of a quadratic space with an isometry.  This is proved for semi-simple isometries by \cite[Lemma 1.2]{Mil}. In order to establish this result for general cases, we first need the following lemma.

\begin{lem}\label{basis} Let $p(x)$ be a monic irreducible polynomial of degree $2d$ over a field $k$ with $p(0)\neq 0$.  If $l$ is a positive integer and $s(x)=x^{-d}p(x)$, then $$\{ x^i s(x)^j : \ 0\leq i \leq 2d-1, \ 0\leq j \leq l-1 \} $$
is a basis of $k[x]/(p(x)^l)$ as $k$-vector space.  
\end{lem}

\begin{proof}  Since $\dim_k(k[x]/(p(x)^l))=2dl$ as a $k$-vector space, one only needs to show that the given vectors are linearly independent over $k$.  Suppose 
$$ \sum_{i=0}^{2d-1} \sum_{j=0}^{l-1} a_{ij} \bar x^i s(\bar x)^j = 0  \ \ \ \text{in} \ \ \ k[x]/(p(x)^l) $$  
where $a_{ij}\in k$ for $0\leq i \leq 2d-1$ and $ 0\leq j \leq l-1$. Since  the following natural maps give a short exact sequence 
$$ 0\longrightarrow (p(x))/(p(x)^l) \longrightarrow k[x]/(p(x)^l) \longrightarrow k[x]/(p(x)) \longrightarrow 0$$ as $k[x]$-modules, one obtains $a_{i0}=0$ for $0\leq i\leq 2d-1$. This implies that 
$$ s(x) \sum_{i=0}^{2d-1} \sum_{j=1}^{l-1} a_{ij} \bar x^i s(\bar x)^{j-1} =0 \ \ \ \text{in} \ \ \ k[x]/(p(x)^l) . $$
Since 
$$ k[x]/(p(x)^{l-1})  \xrightarrow {\cdot s(x)} (p(x))/(p(x)^l) ; \ \alpha(x) \mapsto \alpha(x) s(x) $$ is an isomorphism of $k[x]$-modules, one concludes 
$$ \sum_{i=0}^{2d-1} \sum_{j=1}^{l-1} a_{ij} \bar x^i s(\bar x)^{j-1} =0 \ \ \ \text{in} \ \ \ k[x]/(p(x)^{l-1}) . $$ Then $a_{ij}=0$ for $0\leq i \leq 2d-1$ and $ 1\leq j \leq l-1$ by induction on $l$ and the result follows. 
\end{proof}

The following construction gives an answer to the realization of  \cite[Theorem 3.3]{Mil}, which can also be viewed as generalization of \cite[Lemma 1.2]{Mil}.

\begin{prop} \label{realization} Let $p(x)$ be a monic irreducible polynomial of degree bigger than 1 over a field $k$ with $char(k)\neq 2$ and $p^*(x)=p(x)$.  Suppose that $M$ is a free $k[x]/(p(x)^l)$-module of rank $n$ with a positive integer $l$ and $F=k[x]/(p(x))$. 

If $(M/p(x)M, h)$ is a hermitian space over $F$ with respect to the involution $\iota$ induced by $x\mapsto x^{-1}$,   there is a non-degenerate inner product $(M, \langle \ , \rangle)$ over $k$ with an isometry $\sigma$ such that the minimal polynomial of $\sigma$ is $p(x)^l$ and 
$$ {\rm Tr}_{F/k}(h(\bar u, \bar v))= \langle (\sigma^{-1} p(\sigma))^{l-1} u, v \rangle , \ \ \forall u, v \in M .$$
\end{prop}
\begin{proof} Let $\bar v_1, \cdots, \bar v_n$ be an orthogonal basis of $(M/p(x)M, h)$. Since $M$ is a free $k[x]/(p(x)^l)$-module of rank $n$, one has 
$$ M = (k[x]/(p(x)^l) )v_1 \oplus \cdots \oplus (k[x]/(p(x)^l)) v_n .$$
By Lemma \ref{basis}, each $k$-vector subspace $(k[x]/(p(x)^l) )v_s$ has a basis $$\{ x^i s(x)^j v_s: \ 0\leq i \leq 2d-1, \ 0\leq j \leq l-1 \} $$ over $k$ for $1\leq s\leq n$. Define an inner product on this subspace satisfying 
\begin{equation} \label{sym}  \langle x^{i}s(x)^j v_s, x^{i'}s(x)^{j'} v_s \rangle_s = \langle x^{|i-i'|} s(x)^{j+j'} v_s, v_s \rangle_s \end{equation}
and 
 \begin{equation} \label{value} \langle x^i s(x)^j v_s, v_s \rangle_s= \begin{cases}  {\rm Tr}_{F/k} (x^{i} h(\bar v_s, \bar v_s) ) \ \ \ & j=l-1, \ 0\leq i< d \\  
0 \ \ \ & j \neq l-1, \ 0\leq i<d \end{cases} \end{equation} 
 for $1\leq s\leq n$.  Since 
 $$ \langle x^i s(x)^j v_s, v_s \rangle_s =  \langle x^{-i} s(x)^j v_s, v_s \rangle_s $$ by (\ref{sym}), 
 one obtains 
 $$ \langle x^i s(x)^j v_s, v_s \rangle_s = \frac{1}{2}  \langle (x^i+x^{-i} )s(x)^j v_s, v_s \rangle_s . $$
 When $i\geq d$,  there are polynomials $\alpha(x)$ and $\beta(x)$ with $\deg (\alpha(x) )< i$ and $\deg(\beta(x)) < d$ in $k[x]$ such that 
 $$ x^i+x^{-i} = \alpha(x+x^{-1}) s(x) + \beta(x+x^{-1}) .$$ Then the following inductive formula
 $$ \langle x^i s(x)^j v_s, v_s \rangle_s = \frac{1}{2} \langle \alpha(x+x^{-1}) s(x)^{j+1} v_s, v_s \rangle_s +\frac{1}{2} \langle \beta(x+x^{-1}) s(x)^j v_s, v_s\rangle_s $$
 gives the final value by (\ref{value}). Since 
 $${\rm Tr}_{F/k} (x^{i} h(\bar v_s, \bar v_s) ) = {\rm Tr}_{F/k} (x^{-i} \overline{h(\bar v_s, \bar v_s)} ) = {\rm Tr}_{F/k} (x^{-i} h(\bar v_s, \bar v_s) ), $$ one also has 
 $$\begin{aligned}  & {\rm Tr}_{F/k} (x^{i} h(\bar v_s, \bar v_s) ) =   \frac{1}{2} {\rm Tr}_{F/k} ((x^{i}+x^{-i}) h(\bar v_s, \bar v_s) ) \\
 = & \frac{1}{2} {\rm Tr}_{F/k} (\beta(x+x^{-1})h(\bar v_s, \bar v_s) ). \end{aligned} $$ for $i\geq d$. This extends the first part of (\ref{value}) as 
 $$\langle x^i s(x)^{l-1} v_s, v_s \rangle_s=  {\rm Tr}_{F/k} (x^{i} h(\bar v_s, \bar v_s) ) $$ for all $i$ and
 $$\langle s(x)^{l-1} u, v \rangle_s=  {\rm Tr}_{F/k} (h(\bar u, \bar v) ),   \ \ \ \forall u, v\in (k[x]/(p(x)^l) )v_s $$ by a direct computation.  By ordering the basis as 
 $$ v_s, x v_s, \cdots, x^{2d-1} v_s;$$
 $$ s(x)v_s, xs(x)v_s, \cdots, x^{2d-1} s(x)v_s; $$
 $$\cdots \cdots \cdots $$
 $$s(x)^{l-1}v_s, x s(x)^{l-1}v_s, \cdots, x^{2d-1}s(x)^{l-1}v_s  , $$
 the determinant of the inner product $\langle \ , \rangle_s$ with respect to this basis is equal to 
 $$ \pm \det( {\rm Tr}_{F/k} (h(x^i \bar v_s, x^j \bar v_s) ))_{0\leq i, j \leq 2d-1} ^l . $$ Since $$\det( {\rm Tr}_{F/k} (h(x^i \bar v_s, x^j \bar v_s) ))_{0\leq i, j \leq 2d-1} $$ is equal to the determinant of the quadratic space ${\rm Tr}_{F/k} \circ h|_{F \bar v_s}$, one concludes that 
 $$\det( {\rm Tr}_{F/k} (h(x^i \bar v_s, x^j \bar v_s) ))_{0\leq i, j \leq 2d-1}  \neq 0 $$ by \cite[Lemma 1.2]{Mil} and the quadratic space $((k[x]/(p(x)^l) )v_s, \ \langle \ , \rangle_s)$ is non-degenerate.  Then 
 $$\langle \ , \rangle = \langle \ , \rangle_1 \perp \cdots \perp \langle \ , \rangle_n $$ is a non-degenerate inner product on $M$.  It can be verify directly that the map 
 $$\sigma: M \longrightarrow M; \ m \mapsto x m \ \ \ \ \forall m\in M$$ is an isometry of $(M, \langle \ , \rangle )$ as desired. 
\end{proof}

One has also the following realization of \cite[Theorem 3.4]{Mil}.

\begin{prop}\label{real-odd}  Let $M$ be a free $k[x]/(q(x)^l)$-module of rank $n$ with an odd integer $l>0$, where $q(x)=(x+1)$ or $(x-1)$.

If $(M/q(x) M, b)$ is a non-degenerate quadratic space over $k$ with $char(k)\neq 2$, there is a non-degenerate quadratic space $(M, \langle \ , \rangle)$ over $k$ with an isometry $\sigma$ such that the minimal polynomial of $\sigma$ is $q(x)^l$ and 
$$ b (\bar u, \bar v) = \langle (\sigma- \sigma^{-1})^{l-1} u, v \rangle , \ \ \forall u, v \in M .$$
\end{prop}

\begin{proof} Let $\bar v_1, \cdots, \bar v_n$ be an orthogonal basis of $(M/q(x)M, b)$. Since $M$ is a free $k[x]/(q(x)^l)$-module of rank $n$, one has 
$$ M = (k[x]/(q(x)^l) )v_1 \oplus \cdots \oplus (k[x]/(q(x)^l)) v_n .$$
Since each $k$-vector subspace $(k[x]/(p(x)^l) )v_s$ has a basis $$\{ (x-x^{-1})^i v_s: \ 0 \leq i \leq l-1 \} $$ over $k$ for $1\leq s\leq n$, one defines an inner product on this subspace by
\begin{equation} \label{inner}  \langle (x-x^{-1})^{i} v_s, (x-x^{-1})^{j}v_s \rangle_s = (-1)^j \langle (x-x^{-1})^{i+j}  v_s, v_s \rangle_s \end{equation}
and 
 \begin{equation} \label{pm-value} \langle (x-x^{-1})^i v_s, v_s \rangle_s= \begin{cases} b(\bar v_s, \bar v_s)  \ \ \ & i=l-1\\  
0 \ \ \ & \text{otherwise} \end{cases} \end{equation} 
 for $1\leq s\leq n$.  Since the determinant of the inner product $\langle \ , \rangle_s$ with respect to this basis is equal to $\pm b(\bar v_s, \bar v_s)^l \neq 0$, the quadratic space $(M, \langle \ , \rangle)$ is non-degenerate where $\langle \ , \rangle = \langle \ , \rangle_1 \perp \cdots \perp \langle \ , \rangle_n$.  

Write 
$$ x= \sum_{i=0}^{l-1} \alpha_i (x-x^{-1})^i \ \ \ \text{in} \ \ \ k[x]/(p(x)^l) $$  where $\alpha_i \in k$ for $0\leq i\leq l-1$.  Then 
$$ x= \sum_{i=0}^{l-1} \alpha_i (x-x^{-1})^i +(x-x^{-1})^r x^a g(x) $$ where $r, a\in \Bbb Z$ with $r\geq l$ and $g(x)\in k[x]$ with $(g(x), q(x))=1$. Replacing $x$ with $x^{-1}$ in the above identity, one obtains 
 $$x^{-1} = \sum_{i=0}^{l-1} (-1)^i \alpha_i (x-x^{-1})^i + (x-x^{-1})^r x^b h(x) $$ 
where $b\in \Bbb Z$ and $h(x)\in k[x]$ with $(q(x), h(x))=1$. Then 
$$\langle (x-x^{-1})^i x v_s, v_s \rangle_s= \alpha_{l-1-i} b(\bar v_s, \bar v_s)= \langle (x-x^{-1})^i v_s, x^{-1} v_s \rangle_s  $$
by (\ref{pm-value}). This implies that the map 
 $$\sigma: M \longrightarrow M; \ m \mapsto x m \ \ \ \ \forall m\in M$$ is an isometry of $(M, \langle \ , \rangle )$ as desired. 
\end{proof}

In order to give a proof of Theorem \ref{main}, one needs the similar result as the remark after \cite[Theorem 3.3]{Mil} for Case 2) in \cite{Mil}.

\begin{prop} \label{pm 1} Let $(V,  \langle \  , \  \rangle)$ be a non-degenerate quadratic space over a field of $k$ with $char(k)\neq 2$ and $f(x)$ be the characteristic polynomial of $\sigma\in {\rm O}(V)$. Suppose that $q(x)=(x+ 1)$ or $(x-1)$ is a factor of $f(x)$ and 
$$ V_{q(x), \sigma } = V_{q(x), \sigma }^{(1)} \perp \cdots \perp V_{q(x), \sigma }^{(l)} \perp \cdots $$
where  $(V_{q(x), \sigma}^{(l)}, \langle \ , \rangle)$ is non-degenerate and free as $k[x]/(q(x)^l)$-module by \cite[Theorem 3.2]{Mil}.

When $l$ is even, then $V_{q(x), \sigma }^{(l)}$ is an orthogonal sum of hyperbolic planes. 

When $l$ is odd, then $V_{q(x), \sigma }^{(l)}= W\perp H$ where 
$$ (W, \langle \  , \rangle) \cong (V_{q(x), \sigma }^{(l)}/q(x)V_{q(x), \sigma }^{(l)}, (-1)^{\frac{l-1}{2}} b_\sigma^{(l)}) $$ 
defined by (\ref{bil}) and $H$ is an orthogonal sum of hyperbolic planes. 
\end{prop} 

\begin{proof}  When $l$ is even, then
$$\dim_k (q(x)^{\frac{l}{2}} V_{q(x), \sigma }^{(l)})= \frac{1}{2} \dim_k (V_{q(x), \sigma }^{(l)}) . $$ 
Since $q(x)^{\frac{l}{2}} V_{q(x), \sigma }^{(l)}$ is isotropic, one has that $V_{q(x), \sigma }^{(l)}$ is an orthogonal sum of hyperbolic planes as desired. 

When $l$ is odd, then
 $$ q(x)^{\frac{l+1}{2}} V_{q(x), \sigma}^{(l)}  \subset  q(x)^{\frac{l-1}{2}} V_{q(x), \sigma}^{(l)} \subseteq (q(x)^{\frac{l+1}{2}} V_{q(x), \sigma}^{(l)} )^{\perp} .$$ Since 
$$ \dim_k (q(x)^{\frac{l+1}{2}} V_{q(x), \sigma}^{(l)}) + \dim_k (q(x)^{\frac{l+1}{2}} V_{q(x), \sigma}^{(l)} )^{\perp} = \dim_k (V_{q(x), \sigma }^{(l)}) $$
by \cite[42:6]{OM}, one concludes that 
$$q(x)^{\frac{l-1}{2}} V_{q(x), \sigma}^{(l)} = (q(x)^{\frac{l+1}{2}} V_{q(x), \sigma}^{(l)} )^{\perp} .$$ 
Since $q(x)^{\frac{l+1}{2}} V_{q(x), \sigma}^{(l)}$ is isotropic, there is an orthogonal sum $H$ of hyperbolic planes such that 
$$ q(x)^{\frac{l+1}{2}} V_{q(x), \sigma}^{(l)} \subset H \ \ \ \text{and} \ \ \ V_{q(x), \sigma }^{(l)}= H \perp H^{\perp} . $$ 
Therefore 
$$q(x)^{\frac{l-1}{2}} V_{q(x), \sigma}^{(l)} = (q(x)^{\frac{l+1}{2}} V_{q(x), \sigma}^{(l)} )^{\perp} = q(x)^{\frac{l+1}{2}} V_{q(x), \sigma}^{(l)}  \perp H^{\perp} .$$
This induces a canonical isomorphism of $k$-vector spaces
$$ V_{q(x), \sigma}^{(l)}/q(x)V_{q(x), \sigma}^{(l)} \xrightarrow{\Delta^{\frac{l-1}{2}} \cdot }  q(x)^{\frac{l-1}{2}} V_{q(x), \sigma}^{(l)} / q(x)^{\frac{l+1}{2}} V_{q(x), \sigma}^{(l)} \cong H^{\perp}$$ where $\Delta=\sigma -\sigma^{-1}$.  Since $\langle \Delta u, v \rangle = - \langle u, \Delta v \rangle $ and 
$$ b_{\sigma}^{(l)} (\bar u, \bar v)= \langle \Delta^{l-1} u, v \rangle = (-1)^{\frac{l-1}{2}} \langle \Delta^{\frac{l-1}{2}} u ,  \Delta^{\frac{l-1}{2}} v \rangle, \ \ \forall u, v\in V_{q(x), \sigma}^{(l)} $$
 by (\ref{bil}), the canonical isomorphism is also an isomorphism of the corresponding quadratic spaces up to $(-1)^{\frac{l-1}{2}}$ as desired. 
\end{proof} 

As an application, one has the following sufficient conditions for Theorem \ref{main} over any field $k$ with $char(k)\neq 2$. 

\begin{cor} \label{hyper}  Let $(V,  \langle \  , \  \rangle)$ be a non-degenerate quadratic space over a field of $k$ with $char(k)\neq 2$ and $f(x)$ be the characteristic polynomial of $\sigma\in {\rm O}(V)$  with the factorization (\ref{factor}).  Write 
$$ m_0= \sharp \{ 2l+1 : V_{x+1, \sigma}^{(2l+1)} \neq 0  \} + \sharp \{ 2l+1: V_{x-1, \sigma}^{(2l+1)} \neq 0 \} . $$

If $m_0\leq 1$ and $m_1=0$, then any $\tau\in {\rm O}(V)$ which is conjugate to $\sigma$ in ${\rm GL}(V)$ is conjugate to $\sigma$ in ${\rm O}(V)$.
\end{cor}

\begin{proof}  By \cite[Lemma 3.1]{Mil}, one has
$$ \begin{aligned} V =  \ & V_{x+1, \sigma} \perp V_{x-1, \sigma}  \underset{1\leq j\leq m_2} \perp (V_{\tilde{p}_ j(x), \sigma}\oplus V_{\tilde{p}_ j^*(x), \sigma}) \\
= \ & V_{x+1, \tau} \perp V_{x-1, \tau}  \underset{1\leq j\leq m_2} \perp (V_{\tilde{p}_ j(x), \tau}\oplus V_{\tilde{p}_ j^*(x), \tau})) . \end{aligned} $$
  Since $\sigma$ and $\tau$ are conjugates in ${\rm GL}(V)$, the orthogonal splittings 
$$ V_{x\pm1, \sigma} = V_{x\pm 1, \sigma }^{(1)} \perp \cdots \perp V_{x\pm 1, \sigma }^{(l)} \perp \cdots $$ and
$$ V_{x\pm1, \tau} = V_{x\pm 1, \tau}^{(1)} \perp \cdots \perp V_{x\pm 1, \tau}^{(l)} \perp \cdots $$
 satisfy that  $V_{x\pm 1, \sigma }^{(l)}$ and $V_{x\pm 1, \tau}^{(l)}$ are  free $k[x]/(x\pm 1)^l$-modules with the same rank by \cite[Theorem 3.2]{Mil}.  The structure theorem of finitely generated modules also implies that  $V_{\tilde{p}_j(x), \sigma}  \cong V_{\tilde{p}_j(x), \tau} $
as $k[x]$-modules.  Then there is 
an isomorphism of quadratic spaces
$$ \phi_j:  (V_{\tilde{p}_j(x), \sigma} \oplus V_{\tilde{p}_j^*(x), \sigma}, \langle, \rangle)  \xrightarrow{\cong}   (V_{\tilde{p}_j(x), \tau} \oplus V_{\tilde{p}_j^*(x), \tau}, \langle, \rangle) $$ such that $\phi_j \circ \sigma= \tau \circ \phi_j$ for $1\leq j\leq m_2$.

Since $m_0\leq 1$, one can simply assume that there is at most one $l_0$ such that $$V_{x-1, \sigma }^{(2l_0+1)} \neq 0 \ \ \ \text{and} \ \ \  V_{x-1, \tau}^{(2l_0+1)} \neq 0 . $$ 
Applying Proposition \ref{pm 1},  one obtains 
$$(V_{x-1, \sigma }^{(2l_0+1)}/ (x-1)V_{x-1, \sigma }^{(2l_0+1)}, b_\sigma^{(2l_0+1)}) \cong  (V_{x-1, \tau}^{(2l_0+1)}/(x-1)V_{x-1, \tau}^{(2l_0+1)}, b_\tau^{(2l_0+1)}) $$ by Witt's theorem \cite[42:16. Theorem]{OM}.  
Since the non-degenerate sympletic spaces given by $b_\sigma^{(2l)}$ and $b_\tau^{(2l)}$ are automotically isomorphic for all $l$, there is an isomorphism of quadratic spaces
$$ \rho_{\pm}  : \ (V_{x\pm 1, \sigma}, \langle \  , \  \rangle ) \xrightarrow{\cong} (V_{x\pm 1, \tau}, \langle \  , \  \rangle) $$
 such that $\rho_{\pm} \circ \sigma= \tau \circ \rho_{\pm}$ by \cite[Theorem 3.4]{Mil}. Then 
 $$ \phi=\rho_+\perp \rho_- \perp \phi_1\perp \cdots \perp \phi_{m_2} \in {\rm O}(V) \ \ \ \text{and} \ \ \ \tau=\phi \circ \sigma \circ \phi^{-1} $$ as desired.
\end{proof}

\section{Proof of main theorem}

In this section, we assume that $k$ is a non-archimedean local field.

\begin{thm}  Let $V$ be a non-degenerated quadratic space over a non-archimedean local field $k$ with $char(k)\neq 2$. Suppose $\sigma\in {\rm O}(V)$ and $f(x)$ is the characteristic polynomial of $\sigma$ with the factorization (\ref{factor}). Write 
$$ m_0= \sharp \{ 2l+1 : V_{x+1, \sigma}^{(2l+1)} \neq 0  \} + \sharp \{ 2l+1: V_{x-1, \sigma}^{(2l+1)} \neq 0 \} . $$

Then any isometry in ${\rm O}(V)$ which is conjugate to $\sigma$ in ${\rm GL}(V)$ is conjugate to $\sigma$ in ${\rm O}(V)$ if and only if  one of the following conditions holds

(i) $m_0\leq 1$ and $m_1=0$.

(ii) $m_0=0$ and $m_1=1$ and $V_{p_1(x), \sigma}=V_{p_1(x), \sigma}^{(2l_1+1)}$ for some integer $l_1\geq 0$.   

(iii) $m_0=m_1=1$ and $V_{p_1(x), \sigma}=V_{p_1(x), \sigma}^{(2l_1+1)}$ for some integer $l_1\geq 0$ and the unique non-zero quadratic space $$(V_{x\pm 1, \sigma}^{(2l_0+1)}/(x\pm 1)V_{x\pm 1, \sigma}^{(2l_0+1)} , b_\sigma^{(2l_0+1)})$$ defined by (\ref{bil}) is either one dimensional or a hyperbolic plane.
\end{thm}

\begin{proof}   
Let $\tau\in {\rm O}(V)$ such that $\tau$ is conjugate to $\sigma$ in ${\rm GL}(V)$. Then the characteristic polynomial of $\tau$ is also $f(x)$ and 
$$ \begin{aligned} V = & \ V_{x+1, \sigma} \perp V_{x-1, \sigma} \perp ( \underset{1\leq i\leq m_1} \perp  V_{p_i(x), \sigma}) \perp (\underset{1\leq j\leq m_2} \perp (V_{\tilde{p}_ j(x), \sigma}\oplus V_{\tilde{p}_ j^*(x), \sigma})) \\
= & \ V_{x+1, \tau} \perp V_{x-1, \tau} \perp ( \underset{1\leq i\leq m_1} \perp  V_{p_i(x), \tau}) \perp (\underset{1\leq j\leq m_2} \perp (V_{\tilde{p}_ j(x), \tau}\oplus V_{\tilde{p}_ j^*(x), \tau})) \end{aligned}$$
by \cite[Lemma 3.1]{Mil}.  Moreover, there are orthogonal splittings 
$$ V_{x\pm1, \sigma} = V_{x\pm 1, \sigma }^{(1)} \perp \cdots \perp V_{x\pm 1, \sigma }^{(l)} \perp \cdots $$
and 
$$ V_{x\pm1, \tau} = V_{x\pm 1, \tau}^{(1)} \perp \cdots \perp V_{x\pm 1, \tau}^{(l)} \perp \cdots $$
where $V_{x\pm 1, \sigma }^{(l)}$ and $V_{x\pm 1, \tau}^{(l)}$ are  free $k[x]/(x\pm 1)^l$-modules with the same rank by \cite[Theorem 3.2]{Mil}. Similarly, 
$$ V_{p_i(x), \sigma} = V_{p_i(x), \sigma}^{(1)} \perp \cdots \perp V_{p_i(x), \sigma}^{(l)}\perp \cdots $$
and 
$$ V_{p_i(x), \tau} = V_{p_i(x), \tau}^{(1)} \perp \cdots \perp V_{p_i(x), \tau}^{(l)}\perp \cdots $$
where $V_{p_i(x), \sigma}^{(l)}$ and $V_{p_i(x), \tau}^{(l)}$ are free $k[x]/(p_i(x)^l)$-modules of the same rank by \cite[Theorem 3.2]{Mil} for $1\leq i\leq m_1$. By the same argument as those in Corollary \ref{hyper},  there is an isomorphism of quadratic spaces
$$ \phi_j : \ (V_{\tilde{p}_ j(x), \sigma}\oplus V_{\tilde{p}_ j^*(x), \sigma}, \langle \  , \  \rangle ) \xrightarrow{\cong} (V_{\tilde{p}_ j(x), \tau}\oplus V_{\tilde{p}_ j^*(x), \tau}, \langle \  , \  \rangle) $$
such that $\phi_j \circ \sigma= \tau \circ \phi_j$ for $1\leq j\leq m_2$. 

\medskip

{\bf Sufficiency}. One only needs to treat (ii) and (iii) by Corollary \ref{hyper}.

\medskip

If $m_0=0$ and $m_1=1$ and $V_{p_1(x), \sigma}=V_{p_1(x), \sigma}^{(2l_1+1)}$,  then  $V_{p_1(x), \tau}=V_{p_1(x), \tau}^{(2l_1+1)}$ and 
$$ ( V_{p_1(x), \sigma}^{(2l_1+1)}/p_1(x) V_{p_1(x), \sigma}^{(2l_1+1)}, \ b_\sigma^{(2l_1+1)} )\cong (V_{p_1(x), \tau}^{(2l_1+1)}/ p_1(x) V_{p_1(x), \tau}^{(2l_1+1)}, \ b_{\tau}^{(2l_1+1)}) $$
 where 
\begin{equation} \label{bil-s} b_\sigma^{(2l_1+1)} (\bar u, \bar v) = \langle (\sigma^{-d} p_1(\sigma))^{2l_1} u, v \rangle , \ \ \ \forall  u, v \in V_{p_1(x), \sigma}^{(2l_1+1)} \end{equation}
and 
\begin{equation} \label{bil-t}  b_\tau^{(2l_1+1)} (\bar u, \bar v) = \langle (\tau^{-d} p_1(\tau) )^{2l_1} u ,  v \rangle , \ \ \ \forall  u, v \in V_{p_1(x), \tau}^{(2l_1+1)} \end{equation}  with $\deg (p_1(x))=2d$ by the remark after \cite[Theorem 3.3]{Mil} and Witt's theorem \cite[42:16. Theorem]{OM}. Let 
$$(V_{p_1(x), \sigma}^{(2l_1+1)}/p_1(x) V_{p_1(x), \sigma}^{(2l_1+1)}, h_\sigma^{(2l_1+1)}) \ \text{ and } \ (V_{p_1(x), \tau}^{(2l_1+1)}/ p_1(x) V_{p_1(x), \tau}^{(2l_1+1)}, h_\tau^{(2l_1+1)})$$ be the corresponding hermitian spaces such that 
$$ b_\sigma^{(2l_1+1)}= {\rm Tr}_{F/k}  \circ h_\sigma^{(2l_1+1)} \ \ \ \text{and} \ \ \ b_{\tau}^{(2l_1+1)} = {\rm Tr}_{F/k} \circ h_\tau^{(2l_1+1)} $$
with $F=k[x]/(p_1(x))$ by \cite[Lemma 1.1]{Mil}.  One has an isomorphism of hermitian spaces
$$(V_{p_1(x), \sigma}^{(2l_1+1)}/p_1(x) V_{p_1(x), \sigma}^{(2l_1+1)}, h_\sigma^{(2l_1+1)}) \cong (V_{p_1(x), \tau}^{(2l_1+1)}/ p_1(x) V_{p_1(x), \tau}^{(2l_1+1)}, h_\tau^{(2l_1+1)})$$
by \cite[Theorem 2.7]{Mil}. 
Then there is an isomorphism of quadratic spaces
$$ \psi_1:  (V_{p_1(x), \sigma}, \langle \ , \rangle ) \xrightarrow{\cong} (V_{p_1(x), \tau}, \langle \ , \rangle ) $$
such that $\psi_1 \circ \sigma= \tau \circ \psi_1$ by \cite[Theorem 3.3]{Mil}. Therefore
$$ \phi = (\rho_+\perp \rho_- \perp \psi_1 \perp \phi_1 \perp \cdots \perp  \phi_{m_2} ) \in {\rm O}(V) $$ and $\tau=\phi \circ \sigma \circ \phi^{-1}$ as desired. 

\medskip

If $m_0=m_1=1$ and $V_{p_1(x), \sigma}=V_{p_1(x), \sigma}^{(2l_1+1)}$ for some integer $l_1\geq 0$ and one simply assume that $$(V_{x-1, \sigma}^{(2l_0+1)}/(x-1)V_{x - 1, \sigma}^{(2l_0+1)} , b_\sigma^{(2l_0+1)})$$  is either one dimensional or a hyperbolic plane, then $$V_{p_1(x), \tau}=V_{p_1(x), \tau}^{(2l_1+1)} \ \ \ \text{ and } \ \ \ 
1\leq \dim_k (V_{x-1, \tau}^{(2l_0+1)}/(x-1)V_{x - 1, \tau}^{(2l_0+1)})\leq 2$$ respectively. Since the determinants of quadratic spaces 
$$ ( V_{p_1(x), \sigma}^{(2l_1+1)}/p_1(x) V_{p_1(x), \sigma}^{(2l_1+1)}, \ b_\sigma^{(2l_1+1)} ) \ \ \text{and} \ \ (V_{p_1(x), \tau}^{(2l_1+1)}/ p_1(x) V_{p_1(x), \tau}^{(2l_1+1)}, \ b_{\tau}^{(2l_1+1)}) $$
defined by (\ref{bil-s}) and (\ref{bil-t}) are equal by \cite[Theorem 2.7]{Mil}, one obtains the determinants of quadratic spaces 
$$(V_{x-1, \sigma}^{(2l_0+1)}/(x-1)V_{x - 1, \sigma}^{(2l_0+1)} , b_\sigma^{(2l_0+1)}) \ \text{and} \ (V_{x-1, \tau}^{(2l_0+1)}/(x-1)V_{x - 1, \tau}^{(2l_0+1)} , b_\tau^{(2l_0+1)})$$
are equal by the remark after \cite[Theorem 3.3]{Mil} and Proposition \ref{pm 1}.  Then 
$$(V_{x-1, \sigma}^{(2l_0+1)}/(x-1)V_{x - 1, \sigma}^{(2l_0+1)} , b_\sigma^{(2l_0+1)}) \cong  (V_{x-1, \tau}^{(2l_0+1)}/(x-1)V_{x - 1, \tau}^{(2l_0+1)} , b_\tau^{(2l_0+1)}) $$
as quadratic spaces by the assumption and \cite[42:9]{OM}.  One further obtains an isomorphism of quadratic spaces 
$$( V_{p_1(x), \sigma}^{(2l_1+1)}/p_1(x) V_{p_1(x), \sigma}^{(2l_1+1)}, \ b_\sigma^{(2l_1+1)} ) \cong (V_{p_1(x), \tau}^{(2l_1+1)}/ p_1(x) V_{p_1(x), \tau}^{(2l_1+1)}, \ b_{\tau}^{(2l_1+1)}) $$
by Proposition \ref{pm 1} and the remark after \cite[Theorem 3.3]{Mil} and Witt's theorem \cite[42:16. Theorem]{OM}. The result follows from the same arguments as above and those in Corollary \ref{hyper}.

\medskip

{\bf Necessity}. Suppose $m_1\geq 1$ and there is a positive integer $l_1$ such that $V_{p_1(x), \sigma}^{(2l_1)} \neq 0$. Let $$( V_{p_1(x), \sigma}^{(2l_1)} / p_1(x) V_{p_1(x), \sigma}^{(2l_1)} ,  h^{(2l_1)} )$$ be the corresponding hermitian space given in \cite[Theorem 3.3]{Mil}. Since $k[x]/(p_1(x))$ is a non-archimedean local field, there is another hermitian space
$$(V_{p_1(x), \sigma}^{(2l_1)} / p_1(x) V_{p_1(x), \sigma}^{(2l_1)} ,  h_0^{(2l_1)} )$$
with a different determinant by \cite[Example 3, \S 1]{Mil}. By Proposition \ref{realization},  there is a non-degenerate quadratic space $(V_{p_1(x), \sigma}^{(2l_1)} , \langle \ , \rangle ')$ with an isometry $\tau_0$ such that $(V_{p_1(x), \sigma}^{(2l_1)} / p_1(x) V_{p_1(x), \sigma}^{(2l_1)} , \ h_0^{(2l_1)} )$ is a complete invariant of this pair. By the remark after \cite[Theorem 3.3]{Mil}, one has 
$$(V_{p_1(x), \sigma}^{(2l_1)} , \langle \ , \rangle ) \cong (V_{p_1(x), \sigma}^{(2l_1)} , \langle \ , \rangle '). $$
Replacing $\sigma |_{V_{p_1(x), \sigma}^{(2l_1)}}$ with $\tau_0$ at $V_{p_1(x), \sigma}^{(2l_1)} $ for $\sigma$, one obtains a new isometry $\tau$ of $(V, \langle \ , \rangle)$. Then $\tau$ is conjugate to $\sigma$ in ${\rm GL}(V)$ but not in ${\rm O}(V)$. 

\medskip

Suppose $m_1\geq 1$ and there are non-negative integers $l_1\neq l_2$ such that $V_{p_1(x), \sigma}^{(2l_1+1)} \neq 0$ and $V_{p_1(x), \sigma}^{(2l_2+1)} \neq 0$. Let $$( V_{p_1(x), \sigma}^{(2l_1+1)} / p_1(x) V_{p_1(x), \sigma}^{(2l_1+1)} ,  h^{(2l_1+1)} ) \ \ \ \text{and} \ \ \ ( V_{p_1(x), \sigma}^{(2l_2+1)} / p_1(x) V_{p_1(x), \sigma}^{(2l_2+1)} ,  h^{(2l_2+1)} ) $$ be the corresponding hermitian spaces given in \cite[Theorem 3.3]{Mil} respectively. Then there are hermitian spaces 
 $$( V_{p_1(x), \sigma}^{(2l_1+1)} / p_1(x) V_{p_1(x), \sigma}^{(2l_1+1)} ,  h_0^{(2l_1+1)} ) \ \ \ \text{and} \ \ \ ( V_{p_1(x), \sigma}^{(2l_2+1)} / p_1(x) V_{p_1(x), \sigma}^{(2l_2+1)} ,  h_0^{(2l_2+1)} ) $$
 with different determinants by \cite[Example 3, \S 1]{Mil} respectively. By Proposition \ref{realization},  there are two pairs of a non-degenerate quadratic space with an isometry
  $$((V_{p_1(x), \sigma}^{(2l_1+1)} , \langle \ , \rangle '), \  \tau_1) \ \ \ \text{and} \ \ \  ((V_{p_1(x), \sigma}^{(2l_2+1)} , \langle \ , \rangle '), \ \tau_2) $$ such that $h_0^{(2l_1+1)} $ and $h_0^{(2l_2+1)}$ are the complete invariants respectively.  Since 
 $$(V_{p_1(x), \sigma}^{(2l_1+1)} , \langle \ , \rangle )\perp (V_{p_1(x), \sigma}^{(2l_2+1)} , \langle \ , \rangle) \cong (V_{p_1(x), \sigma}^{(2l_1+1)} , \langle \ , \rangle ') \perp (V_{p_1(x), \sigma}^{(2l_2+1)} , \langle \ , \rangle ')$$ by comparing the Hasse symbols with \cite[58:3 Remark]{OM}, the remark after \cite[Theorem 3.3]{Mil} and \cite[Theorem 2.7]{Mil}, one obtains a new isometry $\tau$ of $(V, \langle \ , \rangle)$ by replacing $$\sigma |_{V_{p_1(x), \sigma}^{(2l_1+1)}} \perp \sigma |_{V_{p_1(x), \sigma}^{(2l_2+1)}} \ \ \ \text{ with } \ \ \ \tau_1\perp \tau_2$$ for $\sigma$. Then $\tau$ is conjugate to $\sigma$ in ${\rm GL}(V)$ but not in ${\rm O}(V)$.

 \medskip
 
 Suppose $m_0\geq 2$. Without loss of generality, we assume that there are non-negative integers $l_{-}$ and $l_{+}$ such that $$V_{x-1, \sigma}^{(2l_{-}+1)} \neq 0 \ \ \ \text{ and } \ \ \ V_{x+1, \sigma}^{(2l_{+}+1)} \neq 0 . $$ Write 
 $$ W_-= V_{x-1, \sigma}^{(2l_{-}+1)}/(x-1)V_{x-1, \sigma}^{(2l_{-}+1)}  \ \ \text{and}  \ \ W_+= V_{x+1, \sigma}^{(2l_{+}+1)}/(x+1)V_{x+1, \sigma}^{(2l_{+}+1)} .$$

In the following arguments, we'll construct two quadratic spaces $(W_-, b_-)$ and $(W_+, b_+)$ satisfying
 $$(W_-, b_-) \not \cong (W_-,  b_{\sigma}^{(2l_-+1)})  \ \ \ \text{ and } \ \ \ (W_+, b_+) \not \cong (W_+,  b_{\sigma}^{(2l_++1)}) $$ 
 defined by $(\ref{bil})$. By Proposition \ref{real-odd},  one obtains two pairs of a quadratic space with an isometry 
  $$((V_{x-1, \sigma}^{(2l_-+1)} , \langle \ , \rangle '), \  \tau_-) \ \ \ \text{and} \ \ \  ((V_{x+1, \sigma}^{(2l_++1)} , \langle \ , \rangle '), \ \tau_+) $$ 
 such that $(W_-, b_-)$ and $(W_+, b_+)$ are the complete invariants respectively. We further need 
  \begin{equation} \label{twist}   \begin{aligned} & (V_{x-1, \sigma}^{(2l_-+1)} , \langle \ , \rangle )\perp (V_{x+1, \sigma}^{(2l_+ +1)} , \langle \ , \rangle)  \\
  \cong \ & (V_{x-1, \sigma}^{(2l_-+1)} , \langle \ , \rangle ') \perp (V_{x+1, \sigma}^{(2l_+ +1)} , \langle \ , \rangle ') \end{aligned}  \end{equation}
 in the above construction. A new isometry $\tau$ of $(V, \langle \ , \rangle)$ is obtained by replacing $$\sigma|_{V_{x-1, \sigma}^{(2l_-+1)}} \perp \sigma |_{V_{x+1, \sigma}^{(2l_+ +1)}} \ \ \ \text{ with } \ \ \ \tau_-\perp \tau_+$$ for $\sigma$. Then $\tau$ is conjugate to $\sigma$ in ${\rm GL}(V)$ but not in ${\rm O}(V)$.

If $ \min \{ \dim_k (W_-),  \dim_k (W_{+})\} \geq 2 $
 and neither of $(W_-,  b_{\sigma}^{(2l_-+1)})$ nor $(W_+,  b_{\sigma}^{(2l_++1)})$ is a hyperbolic plane, we choose
 $ (W_-, \ b_-)$ and $(W_+, \ b_+)$ with the same determinants as $b_{\sigma}^{(2l_-+1)}$ and $b_{\sigma}^{(2l_++1)}$ but the different Hasse symbols respectively by \cite[63:22.  Theorem]{OM}. The property (\ref{twist}) is satisfied by comparing the Hasse symbols with \cite[58:3 Remark]{OM} and Proposition \ref{pm 1}. 

 If $\dim_k (W_-)=\dim_k (W_{+})=2$ and one of $$ (W_-,  b_{\sigma}^{(2l_-+1)}) \ \ \ \text{ and } \ \ \ (W_+,  b_{\sigma}^{(2l_++1)})$$ is a hyperbolic plane with $(W_-, \ b_{\sigma}^{(2l_-+1)}) \not \cong (W_+, \ b_{\sigma}^{(2l_++1)}) $,  one can assume that $(W_-, \ b_{\sigma}^{(2l_-+1)})$ is a hyperbolic plane. Choose $(W_-, b_-)$ to be $$(W_-, b_-) \cong (W_+, (-1)^{(l_+ - l_-)} b_\sigma^{(2l_++1)})$$ and $(W_+, b_+)$ to be a hyperbolic plane. The property (\ref{twist}) follows from Proposition \ref{pm 1}. 
 
 If both $(W_-, b_{\sigma}^{(2l_-+1)})$ and $(W_+, b_{\sigma}^{(2l_++1)})$ are hyperbolic planes, then one can choose $(W_-, b_-)$ and 
 $(W_+, b_+)$ such that 
 $$(W_-, (-1)^{l_-}b_-) \cong (W_+, (-1)^{l_+} b_+ )$$ which is not a hyperbolic plane but represents all units of $k$ by \cite[63:15. Example (i)]{OM}.  Then the property (\ref{twist}) follows from Proposition \ref{pm 1}. 
 
 If $\dim_k (W_-)=\dim_k (W_{+})=1$, one can choose $$\gamma\in N_{k(\sqrt{-\delta})/k}(k(\sqrt{-\delta})^\times)\setminus (k^\times)^2$$
 where $\delta$ is the determinant of binary quadratic space $$(W_-, (-1)^{l_-}b_\sigma^{(2l_-+1)}) \perp (W_+, (-1)^{l_+}  b_{\sigma}^{(2l_++1)}) .$$ Define $(W_-, b_-)$ and $(W_+, b_+)$ to be
 $$ (W_-, b_-) \cong (W_-, \gamma b_\sigma^{(2l_-+1)}) \ \ \ \text{and} \ \ \ (W_+, b_+) \cong (W_+,  \gamma b_{\sigma}^{(2l_++1)}) .$$
 Then the property (\ref{twist}) follows from Proposition \ref{pm 1} and \cite[63:15. Example (iii)]{OM}. 
 
 The last possibility is one of $W_-$ and $W_+$ is one dimensional and the other one is two dimensional. Without loss of generality, we assume $\dim_k(W_-)=1$ and $\dim_k(W_+)=2$. 
Choose $$ \xi\in (-1)^{l_-+l_+} b_{\sigma}^{(2l_++1)}(W_+) \setminus b_\sigma^{(2l_-+1)}(W_-) $$ by \cite[63:15. Example (ii)]{OM}. Then $$(W_+, b_{\sigma}^{(2l_++1)})\cong <(-1)^{l_-+l_+} \xi> \perp <(-1)^{l_-+l_+} \eta >$$ for some $\eta\in k^\times$. Let 
$$(W_-, b_-) \cong <\xi >  \ \ \text{and}  \ \ (W_+, b_+) \cong <(-1)^{l_-+l_+} \alpha> \perp <(-1)^{l_-+l_+} \eta > $$
where $\alpha$ is the determinant of $(W_-, b_\sigma^{(2l_-+1)})$. Then the property (\ref{twist}) follows from Proposition \ref{pm 1}. 

\medskip

Suppose $m_0=m_1=1$ and $V_{p_1(x), \sigma}=V_{p_1(x), \sigma}^{(2l_1+1)}$ for some integer $l_1\geq 0$ and the unique non-zero quadratic space $$(V_{x\pm 1, \sigma}^{(2l_0+1)}/(x\pm 1)V_{x\pm 1, \sigma}^{(2l_0+1)} , b_\sigma^{(2l_0+1)})$$ defined by (\ref{bil}) is more than one dimensional and is not a hyperbolic plane. By \cite[63:22.  Theorem]{OM}, there is a quadratic space $$(V_{x\pm 1, \sigma}^{(2l_0+1)}/(x\pm 1)V_{x\pm 1, \sigma}^{(2l_0+1)} , b_{\pm})$$ with the same determinants as $b_{\sigma}^{(2l_0+1)}$  but the different Hasse symbol. Applying Proposition \ref{real-odd},  one obtains a pairs of a quadratic space with an isometry 
  $((V_{x\pm 1, \sigma}^{(2l_0+1)} , \langle \ , \rangle '), \  \tau_\pm) $ 
 such that $(V_{x\pm 1, \sigma}^{(2l_0+1)}/(x\pm 1)V_{x\pm 1, \sigma}^{(2l_0+1)} , b_{\pm})$ is the complete invariant.
 
 Let $$(V_{p_1(x), \sigma}^{(2l_1+1)}/p(x)V_{p_1(x), \sigma}^{(2l_1+1)}, h_\sigma^{(2l_1+1)})$$ be the corresponding hermitian space given in \cite[Theorem 3.3]{Mil}. Then  there is another hermitian structure
$$(V_{p_1(x), \sigma}^{(2l_1+1)} / p_1(x) V_{p_1(x), \sigma}^{(2l_1+1)} ,  h_0^{(2l_1+1)} )$$
with a different determinant by \cite[Example 3, \S 1]{Mil}.  Applying Proposition \ref{realization}, one obtains a pair $((V_{p_1(x), \sigma}^{(2l_1+1)} , \langle \ , \rangle '), \tau_1)$ of a non-degenerate quadratic space with an isometry such that $(V_{p_1(x), \sigma}^{(2l_1+1)} / p_1(x) V_{p_1(x), \sigma}^{(2l_1+1)} , \ h_0^{(2l_1+1)} )$ is a complete invariant of this pair.
Applying the remark after \cite[Theorem 3.3]{Mil} and Proposition \ref{pm 1}, one has
$$ (V_{p_1(x), \sigma}^{(2l_1+1)}, \langle \ , \rangle) \perp (V_{x\pm 1, \sigma}^{(2l_0+1)}, \langle \ , \rangle) \cong (V_{p_1(x), \sigma}^{(2l_1+1)}, \langle \ , \rangle') \perp (V_{x\pm 1, \sigma}^{(2l_0+1)}, \langle \ , \rangle') $$ as quadratic spaces by comparing the determinants using \cite[Theorem 2.7]{Mil} and the Hasse symbols using \cite[58:3 Remark]{OM}.  
 Replacing $$\sigma |_{V_{p_1(x), \sigma}^{(2l_1+1)}} \perp \sigma |_{V_{x\pm 1, \sigma}^{(2l_0+1)}} \ \ \ \text{ with } \ \ \ \tau_1\perp \tau_{\pm}$$ for $\sigma$, one obtains $\tau\in {\rm O}(V)$  which conjugate to $\sigma$ in ${\rm GL}(V)$ but not in ${\rm O}(V)$ as desired.
\end{proof}

\end{document}